\documentclass[11pt]{amsart}
\usepackage{multicol}
\usepackage{indentfirst}
\usepackage{latexsym}
\usepackage{bm}
\usepackage{graphicx}
\usepackage{subfigure}
\usepackage{float}
\usepackage{verbatim}
\usepackage{comment}
\usepackage[top=1in, bottom=1in, left=1.25in, right=1.25in]{geometry}
\usepackage{epsfig,dsfont,amssymb,amsmath,amsthm,amsfonts,amsbsy,mathrsfs}
\usepackage{amsmath,amssymb,amscd,mathrsfs,slashed,enumerate,url}
\usepackage[numbers]{natbib}         
\usepackage[colorlinks]{hyperref} 
\newcommand{\gpp}{\mathfrak{g}_P}

\newcommand{\slc}{SL(2;\mathbb{C})}

\newcommand{\om}{\omega}

\newcommand{\ti}{\times}
\newcommand{\RP}{\mathbb{R}^+}

\newcommand{\OO}{\mathcal{O}}
\newtheorem{theorem}{Theorem}[section]

\newtheorem{corollary}[theorem]{Corollary}

\newtheorem{definition}[theorem]{Definition}
\newtheorem{example}[theorem]{Example}

\newtheorem{lemma}[theorem]{Lemma}

\newtheorem{proposition}[theorem]{Proposition}
\newtheorem*{remark}{Remark}

\newcommand{\Vol}{\mathrm{Vol}}
\newcommand{\MC}{\mathcal{C}}
\newcommand{\Tr}{\mathrm{Tr}}
\newcommand{\pol}{polyhomogeneous\;}

\newcommand{\ML}{\mathcal{L}}

\newcommand{\MO}{\mathcal{O}}
\newcommand{\st}{\star}
\newcommand{\sta}{{\star_4}}
\newcommand{\we}{\wedge}
\newcommand{\YI}{X\times \mathbb{R}^+}
\newcommand{\SA}{F_A-\Phi\we\Phi}
\newcommand{\SB}{d_A\Phi}
\newcommand{\Ph}{\Phi}
\newcommand{\CS}{\mathrm{CS}}
\newcommand{\py}{\partial_y}

\newcommand{\ft}{\mathfrak{t}}
\newcommand{\fe}{\mathfrak{e}}
\newcommand{\al}{\alpha}

\newcommand{\pa}{\partial}
\newcommand{\MA}{\mathcal{A}}
\newcommand{\MB}{\mathcal{B}}

\newcommand{\ep}{\epsilon}
\newcommand{\si}{\sigma}
\newcommand{\MP}{\mathcal{P}}
\newcommand{\lam}{\lambda}
\newcommand{\mft}{\mathfrak{t}}

\usepackage[utf8]{inputenc}
\usepackage[english]{babel}
\usepackage{fancyhdr}

\begin{document}
	\title[The Expansions of the Nahm Pole Solutions to the Kapustin-Witten Equations]{The Expansions of the Nahm Pole Solutions to the Kapustin-Witten Equations}
	\author{Siqi He} 
	\address{Department of Mathematics, California Institute of Technology\\Pasadena, CA, 91106}
	\email{she@caltech.edu}
	
	\begin{abstract}
		For a 3-manifold $X$ and compact simple Lie group $G$, we study the expansions of \pol Nahm pole solutions to the Kapustin-Witten equations over $X\times (0,+\infty)$. Let $y$ be the coordinate of $(0,+\infty)$, we prove that the sub-leading terms of a \pol Nahm pole solution is smooth to the boundary when $y\to 0$ if and only if $X$ is an Einstein 3-manifold.
	\end{abstract}
	
	\maketitle
	
	\begin{section}{Introduction}
		Witten \cite{witten2011fivebranes} propose a fascinating program for interpreting the Jones polynomial of knots on a 
		$3$-manifold $X$ by counting singular solutions to the Kapustin-Witten equations and when the knot is empty, the counting might lead to new 3-manifold invariants. The singular boundary condition is called the Nahm pole
		boundary condition. We refer \cite{witten2011fivebranes, gaiotto2012knot, Mikhaylov2012, Witten2014LecturesJonesPolynomial, Haydys2015Fukaya, cherkis2015octonions} for a more detailed explanations of this program. 
		
		The main difficulty of this program comes from the singular boundary condition and there are many results towards it. The Nahm pole boundary condition is non-standard elliptic type, the Fredholm theory and regularity of the solutions are understood in \cite{MazzeoWitten2013, MazzeoWitten2017}. In \cite{taubes2018sequences}, Taubes prove a compactness theorem for a sequence of Nahm pole solutions. In this article, we will study the Nahm pole solutions from the expansion side, based on the regularity theorem proved in \cite{MazzeoWitten2013}. 
		
		Now it is time to describe the Kapustin-Witten equations we studied, which is introduced in \cite{KapustinWitten2006}. Let $G$ be a compact simple Lie group, $P$ be a principal $G$ bundle over $X\ti\RP$ where $\RP=(0,+\infty)$, and set $\gpp$ be the adjoint bundle. The Kapustin-Witten equations are the following set of equations for a connection $A$ and $\gpp$-valued 1-form $\Phi$, respectively:
		\begin{equation}
		\begin{split}
		&F_A-\Phi\we\Phi+\st d_A\Phi=0,\\
		&d_{A}^{\st}\Phi=0,
		\label{KW}
		\end{split}
		\end{equation}
		where $F_A$ is the curvature of the connection $A$, $d_A$ is the covariant derivative and $\star$ is the Hodge star operator on 4-manifold.
		
		The regularity theorem in \cite{MazzeoWitten2013} shows that any Nahm pole solution $(A,\Phi)$ to the Kapustin-Witten equations with a gauge fixing condition is \pol along the boundary. We will explicitly introduce it in Section 2. If we denote $y$ to be the coordinate of $\mathbb{R}^+$ and $x$ the local coordinate of $X$, $(A,\Phi)$ is \pol roughly means they have the following expansions:
		
		\begin{equation}
		\begin{split}
		&A\sim\omega+\sum_{i=1}^{+\infty}\sum_{p=1}^{r_i}y^i (\log y)^p a_{i,p},\\
		&\Phi\sim y^{-1}e+\sum_{i=1}^{+\infty}\sum_{p=1}^{r_i}y^i (\log y)^p b_{i,p},
		\label{poly}
		\end{split}
		\end{equation}
		where $e:TX\to\gpp$ is the vierbein form gives an endormorphism of the tangent bundle $TX$ of $X$ and adjoint bundle $\gpp$, $\om$ is the Levi-Civita connection of $X$ under the pullback of $e$. For each $i$, $r_i$ are non-negative integers and $a_{i,p},\;b_{i,p}$ are 1-forms independent of $y$ coordinate and smooth in $x$ direction. We say a function or a form is smooth to the boundary if all the "$\log y$" terms in the expansion disappear.
		
		The vierbein term $e$ depends on a choice of principle embedding $\rho:\mathfrak{su}(2)\to G$ and we write the image of $\rho$ as $\mathfrak{su}(2)_t$. Let $\Omega^1_X(\gpp)$(resp. $\Omega^0_X(\gpp)$) be the $\gpp$-valued 1-form(resp. 0-form) over $X$, the action of $\mathfrak{su}(2)_t$ will give the decomposition $\Omega^1_X(\gpp)=\oplus V_\si$(resp. $\Omega^0(\gpp)=\oplus \tau_\si$), where the $V_\si$ and $\tau_\si$ are irreducible modules and $\si$ take values in positive integers.
		
		We have the following descriptions of the \pol solutions to the Kapustin-Witten equations:
		
		\begin{theorem}
			Let $(A,\Phi)$ be a \pol Nahm pole solution to the Kapustin-Witten equations over $X\times \mathbb{R}^+$. In the temporal gauge where $A$ doesn't have $dy$ component, we write $A=\om+a$, $\Phi=\frac{e}{y}+b+\phi_ydy$. For different $\si$ as above, let $a^{\si}$, $b^{\si}$ be the projection of $a,b$ into the irreducible module $V_{\si}$ and let $\phi_y^\si$ be the projection of $\phi_y$ into the irreducible module $\tau_\si$. Suppose $a^{\si},b^{\si}$,$\phi_y^\si$ have the expansions
			\begin{equation*}
			\begin{split}
			a^{\si}\sim\sum_{i=1}^{+\infty}\sum_{p=0}^{r_i}a^{\si}_{i,p}y^{i}(\log y)^p,\;b^{\si}\sim\sum_{i=1}^{+\infty}\sum_{p=0}^{r_i}b^{\si}_{i,p}y^{i}(\log y)^p,\;\phi^\si_y\sim \sum_{i=1}^{+\infty}\sum_{p=0}^{r_i}(\phi_y^{\si})_{i,p}y^i(\log y)^p.
			\end{split}
			\end{equation*} We write $a^\si_i:=a^\si_{i,0},\;b^{\si}_i:=b^\si_{i,0},\;(\phi_y^\si)_i:=(\phi_y^\si)_{i,0}$. Then we have:
			
			(1) When $\si=1$, $a^1,b^1,\phi_y^1$ have the expansions with leading terms
			\begin{equation*}
			\begin{split}
			&a^1\sim y^2\log y a_{2,1}^1+y^2a_2^1+\MO(y^{\frac{5}{2}}),\\
			&b^1\sim y\log yb_{1,1}^1+yb_1^1+\MO(y^{\frac{5}{2}}),\\
			&\phi_y^1\sim y^2\log y(\phi_y^1)_{2,1}+y^2(\phi_y^1)_2+\MO(y^{\frac52}).
			\end{split}
			\end{equation*}
			
			When $\si>1$, $a^\si,b^\si$ and $\phi_y^\si$ have the expansions with leading terms
			\begin{equation*}
			a^{\si}\sim y^{\si+1}a^{\si}_{\si+1}+\MO(y^{\si+\frac32}),\;b^{\si}\sim y^{\si}b^{\si}_{\si}+\MO(y^{\si+\frac 12}),\;\phi_y^\si\sim y^{\si+1}(\phi_y^{\si})_{\si+1}+\MO(y^{\si+\frac32}).
			\end{equation*}
			
			(2) The expansions of $a,b$ are determined by the coefficients $a^\si_{\si+1},\;b^{\si}_\si,\;(\phi_y^\si)_{\si+1}$. To be explicit, let $(\hat A,\hat \Phi)$ be another solution with the expansion coefficients $\hat a^{\si}_{k,p},\;\hat b^{\si}_{k,p},\;({\hat{\phi_y^\si}})_{k,p}$. If for any $\si$, $a^{\si}_{\si+1}=\hat a^{\si}_{\si+1}$, $ b^{\si}_{\si}=\hat b^{\si}_{\si}$ and $ {(\phi_y^{\si})}_{\si+1}=({\hat{\phi_y^\si}})_{\si+1}$ then  $(A,\Phi)$ and $(\hat{A},\hat{\Phi})$ have the same expansions.
		\end{theorem}
		
		As the leading terms of the $A$'s expansion is the Levi-Civita connection, we can build up the relationship of the geometry of the 3-manifold $X$ and the expansions of the Nahm pole solutions. Under the previous assumptions, we obtain:
		
		\begin{theorem}
			(1) If $b^1_{1,1}=0$, then the expansions of $(A,\Phi)$ don't contains "$\log y$" terms.
			
			(2) $b^1_{1,1}=0$ if and only if $X$ is an Einstein 3-manifold. 
		\end{theorem}
		
		Combining this with the existence results \cite{nahm1980simple, he2015rotationally, Kronheimer2015Personal}, we obtain the following corollary:
		\begin{corollary}
			Over $X\ti\RP$, there exists a Nahm pole solution whose sub-leading term is smooth to the boundary if and only if $X$ is an Einstein 3-manifold.
		\end{corollary}
		
		We can determine the expansions more clearly when $G=SU(2)$ or $SO(3)$:
		
		\begin{theorem}
			When $G=SU(2)$ or $SO(3)$, under the previous assumptions, let $(A,\Phi=\phi+\phi_ydy)$ be a \pol Nahm pole solution, we obtain:
			
			(1) $(A,\Phi=\phi+\phi_ydy)$ has the following expansions:
			\begin{equation*}
			\begin{split}
			&A\sim \om+\sum_{i=1}^{+\infty}\sum_{p=0}^{i}a_{2i,p}y^{2i}(\log y)^p,\;\phi\sim y^{-1}e+\sum_{i=1}^{+\infty}\sum_{p=0}^{i} b_{2i-1,p}y^{2i-1}(\log y)^p,\\
			&\phi_y\sim \sum_{i=1}^{+\infty}\sum_{p=0}^i(\phi_y)_{2i,i} y^{2i}(\log y)^p
			\end{split}
			\end{equation*} where $e$ is the vierbein form and $\om$ under the pull back of $e$ is the Levi-Civita connection of $X$. Remark that here the expansions of $A$ only contains even order terms and the expansions of $\phi,\phi_y$ only contains odd order terms.
			
			(2) If $X$ is an Einstein 3-manifold, then $(A,\Phi=\phi+\phi_ydy)$ has the following expansions
			\begin{equation*}
			\begin{split}
			A\sim \om+\sum_{i=1}^{+\infty}a_{2i}y^{2i},\;\Phi\sim y^{-1}e+\sum_{i=1}^{+\infty}\sum_{p=0} b_{2i-1}y^{2i-1},\;
			\phi_y\sim \sum_{i=1}^{+\infty}\sum_{p=0} (\phi_y)_{2i}y^{2i}.
			\end{split}
			\end{equation*}
		\end{theorem}
		
		\textbf{Acknowledgements:}
		Some parts of the paper are based on the conversations with V.Mikhaylov, the author thanks Victor for his generous and inspiring conversations. The author also greatly thanks C.Manolescu, R.Mazzeo and C.Taubes for their kindness and helpful discussions. This work is finished during the Simons center conference "Gauge Theory and Low Dimensional Topology".
	\end{section}
	
	\begin{section}{The Nahm Pole Solutions}
		In this section, we will summarize some basic properties of the Nahm pole solutions to the Kapustin-Witten equations, largely follows \cite{MazzeoWitten2013}, \cite{henningson2012boundary} and \cite{Mikhaylov17}.

		\begin{subsection}{The Setup of the Nahm Pole Solutions}
			Let $X$ be a smooth 3-manifold with a Riemannian metric, we write $M:=X\times \mathbb{R}^+$ and $y$ be the coordinate of $\mathbb{R}^+$. We choose the product metric over $M$ and volume form $\Vol_X\we dy$, where $\Vol_X$ is the volume form of $X$. Let $P$ be an $G$ bundle over $M$, where $G$ is a compact Lie group with Lie algebra $\mathfrak{g}$. 	Let $(A,\Phi)$ be a solution to the Kapustin-Witten equations \eqref{KW}. 
			
			We first consider the boundary condition on $X\times \{0\}\subset M$. Given a principle embedding $\rho:\mathfrak{su}(2)\to\mathfrak{g}$, for an integer $a=1,2,3$ and a point $x\in X$, let $\{\mathfrak{e}_a\}$ be any unit orthogonal basis of the cotangent bundle of $X$ and $\{\mathfrak{t}_a\}$ be sections of 
			the adjoint bundle $\gpp$ lie in the image of $\rho$ with the 
			relation $[\mathfrak{t}_a,\mathfrak{t
			}_b]=\epsilon_{abc}\mathfrak{t}_c$. We write the vierbein form $e:=\sum_{i=1}^3 \ft_i \fe_i$, where $e$ gives an endormorphism of the tangent bundle $TX$ to the adjoint bundle $\gpp$. The definition of the vierbein form $e$ depends on the choice of $\rho.$ We define the Nahm pole boundary condition as follows:
			
			\begin{definition}
				A solution $(A,\Phi)$ to \eqref{KW} over $M$ is a Nahm pole solution if for any point $x\in X$, there exist 
				$\{\mathfrak{e}_a\}$, $\{\mathfrak{t}_a\}$ as above such that when $y\rightarrow 0$, $(A,\Phi)$ 
				has the following expansions: $A\sim \MO(y^{-1+\ep}),\;\Phi\sim 
				y^{-1}e+\MO(y^{-1+\ep})$, for some constant $\ep>0$.
			\end{definition}
			
			Note that our definitions of Nahm pole are stronger than what is considered in \cite{taubes2018sequences}.
			
			Now, we will introduce some basic terminology of the regularity of a function(or differential forms) over manifold with boundary, largely follows from \cite{Mazzeo1991}. Let $\vec{x}=(x_1,x_2,x_3)$ to be the local coordinates of $X$. Over $X\ti\RP$, for any ($\gpp$-valued) differential form $\alpha$, we say $\alpha$ is conormal if for any $j\geq 0$ and $\vec{k}=(k_1,k_2,k_3)$ with $k_i\geq 0$, $\sup|(y\partial_y)^j \pa_{\vec{x}}^{\vec{k}}\alpha|\leq C_{j\vec{k}}$, where $\pa_{\vec{x}}^{\vec{k}}=\pa_{x_1}^{k_1}\pa_{x_2}^{k_2}\pa_{x_3}^{k_3}$ and the $\sup$ is taken over an open neighborhood of the boundary. 
			We say $\alpha$ is \pol if $\alpha$ is conormal and has an asymptotic expansion $\alpha\sim \sum y^{\gamma_j}(\log y)^p\alpha_{jp}(\vec{x}).$
			Here the exponents $\gamma_j$ lie in some discrete set $E\subset \mathbb{C}$, called the index set of $\alpha$, which has the properties that $\Re\gamma_j\rightarrow \infty$ as $j\rightarrow \infty$, the powers $p$ of $\log y$ are all non-negative integers, and there are only finitely many $\log$ terms accompanying any given $y^{\gamma_j}.$ The notation "$\sim$" means $\sup|\alpha-\sum_{j\leq N}y^{\gamma_j}(\log y)^p\alpha_{jp}|\leq y^{\Re \gamma_{N+1}}(\log y)^q,$ and the corresponding statements must hold for the series obtained by differentiating any finite number of times. We say a function smooth to the boundary if all the "$\log y$" terms in the expansion disappear.

			As the expansion holds for the \pol function after finite time derivation, when we consider two side of an equation, a \pol solution requires that the coefficients for each order $y^k$ should be the same on each side. This terms out split the equations into systems of algebraic equations over the boundary and this will be the starting point of our article.
			
			Now, we will summarize the regularity theorem in \cite{MazzeoWitten2013}:
			\begin{theorem}{\rm{(\cite[Prop 5.3, Prop 5.9, Section 2.3)]{MazzeoWitten2013}}}
				\label{expansionexists}
				For $(A,\Phi)$ a Nahm pole solution to the Kapustin-Witten equations, choose a smooth reference connection $A_0$ and denote $\Phi_0$ the leading term of $\Phi$, if $(A,\Phi)$ satisfies the gauge fixing equation $$d^{\sta}_{A_0}(A-A_0)-\sta[\Phi_0,\sta(\Phi-\Phi_0)]=0,$$ 
				where $\sta$ is the 4-dimensional Hodge star operator, then $(A,\Phi)$ is \pol with the following expansions:
				\begin{equation}
				\begin{split}
				A\sim\omega+\sum_{i=1}^{+\infty}\sum_{p=1}^{r_i}y^i (\log y)^p a_{i,p},\;\Phi\sim y^{-1}e+\sum_{i=1}^{+\infty}\sum_{p=1}^{r_i}y^i (\log y)^p b_{i,p}.
				\label{expansionab}
				\end{split}
				\end{equation}
				Here for each $i$, $r_i$ are finite positive integers and $a_{i,p},\;b_{i,p}$ are 1-forms independent of $y$ coordinate and smooth in $x$ direction, $\om$ is the Levi-Civita connection after pulling back by $e$. 
			\end{theorem}
			\begin{remark}
				We write $a:=A-A_0,\;b:=\Phi-\Phi_0$, the statement that $(A,\Phi)$ is \pol with the gauge fixing condition $d^{\sta}_{A_0}a-\sta[\Phi,\sta b]=0$ is proved in \rm{\cite[Proposition 5.9]{MazzeoWitten2013}} and it also works for many other gauge fixing conditions, for example $A_y=0$ and $d^{\st}_{A_0}a-\st[\Phi,\st b]=0,$ or even $d^{\st}_{A_0}(A-A_0)=0$ where $A_y$ is the $dy$ component of $A$ and $\st$ is the Hodge star operator of $X$. It is straight forward to check that these two gauge fixing conditions will bring the same expansions. The claim that $(A,\Phi)$ has the leading terms $(\omega+y(\log y)^p a_1,y^{-1}e+y(\log y)^p)$ is proved in \rm{\cite[Section 2.3]{MazzeoWitten2013}}.
			\end{remark}
		\end{subsection}
		\begin{subsection}{Cylindrical Boundary Condition}
			Under the previous convention, for $(A,\Phi)$ a Nahm pole solution over $X\times \RP$, under the temporal gauge, we can assume $A$ doesn't have $dy$ component. We write $\Phi=\phi+\phi_ydy$, then the Kapustin-Witten equations reduce to the flow equation:
			\begin{equation}
			\begin{split}
			&\pa_y A=\st d_A\phi+[\phi_y,\phi],\\
			&\pa_y\phi=d_A\phi_y+\st(F_A-\phi\we\phi),\\
			&\pa_y\phi_y=d_A^{\st}\phi.
			\label{flowequation}
			\end{split}
			\end{equation}
			
			At $y=+\infty$, we usually consider the "steady-state" (y-independent) solutions as the boundary condition at infinity. The most wildly considered case is 
			$$
			F_A-\phi\we\phi=0,\;d_A\phi=0,\;d_A^{\st}\phi=0,
			$$
			which imply that the $G^{\mathbb{C}}$ connection $A+i\phi$ is flat.
			If we assume that $A,\phi$ convergence to a flat $G^{\mathbb{C}}$ connection, then a y-independent solution automatically implies $[\phi,\phi_y]=d_A\phi_y=0$. When $A+i\phi$ is reducible, then $\phi_y$ doesn't necessarily vanish.
			
			We have the following well-known vanishing claim for the $\phi_y$ term, for a proof see \cite[Page 36]{taubes2013compactness}, \cite[Corollary 4.7]{He2017}.
			\begin{proposition}
				For $(A,\Phi)$ a Nahm pole solution over $X\times \mathbb{R}^+$, write $\Phi=\phi+\phi_ydy$, then if $\lim_{y\rightarrow +\infty}|\phi_y|_{\MC^0}=0$, then $\phi_y=0.$
				\label{yv}
			\end{proposition}
			
			Therefore, if we assume that $(A,\Phi)$ satisfies the Nahm pole boundary condition and convergence to a flat $G^{\mathbb{C}}$ connection with $\phi_y$ convergences to $0$ when $y\to\infty$, then $\Phi$ will not have $dy$ component. However, as pointing out in \cite{Mikhaylov17}, the $\phi_y$ term appear case is also very interesting. In this article, we don't make any assumption of the limit of the solution at $y\to+\infty$.
			
		\end{subsection}
		\begin{subsection}{Elementary Representation Theory}
			Now we will introduce some representation theory of $\mathfrak{su}(2)$, which largely follows from \cite{MazzeoWitten2013}. 
			
			Let $G$ be a simple Lie group with Lie algebra $\mathfrak{g}$, we consider a principal embedding $\rho:\mathfrak{su}(2)\to\mathfrak{g}$, we call the image $\mathfrak{su}(2)_t$ the principal subalgebra. When $\rho$ is a principal embedding, under the action of $\mathfrak{su}(2)_t,$ $\mathfrak{g}$ decomposes as a direct sum of irreducible modules $\tau_\si$ with dimension $2\si+1$, we write $\mathfrak{g}=\oplus_{\si}\tau_\si$. Here $\si$ are positive integers and $\si=1$  corresponds to the principle subalgebra. For the simple Lie algebra $\mathfrak{g}$, the values of $\si$ are precisely record in \cite{henningson2012boundary}. For example, when $G=SU(N)$, the values of $\si$ are $1,2,3,\cdots,N-1.$ Remark that $\sigma=1$ always exist and it correspond to the principal embedding.
			
			Under the action of $\mathfrak{su}(2)_t$, the decomposition of $\mathfrak{g}$ will automatically induce a decomposition of $\Omega_X^1(\gpp)$, which is the $\gpp$-valued 1-from on $X$. We write $V_\si=\Omega_X^1(\tau_\si)$, then 
			\begin{equation}
			\Omega_X^1(\gpp)=\oplus_\si V_\si,\;\Omega^0(\gpp)=\oplus \tau_\si.
			\label{decomposition1form}
			\end{equation}
			
			We define the projection map $\MP_\si:\Omega^1_X(\gpp)\to V_\si$ and $\MP_\si:\Omega^0_X(\gpp)\to \Omega^0(\tau_\si)$, for $a,b\in\Omega^1_X(\gpp)$, we write $a^\si:=\MP_\si a,\;b^\si:=\MP_\si b$ and $\phi_y^\si:=\MP_\si \phi_y$. The Clebsch-Gordan theorem implies the following proposition:
			\begin{proposition}{\rm{\cite{Kostant59}}}
				\label{ClebschGordan}
				Under the previous assumptions, for $\si_1,\si_2$ are positive integers, then $$\st a^{\si_1}\we b^{\si_2}\in\oplus_{\si=|\si_1-\si_2|}^{\si_1+\si_2}V_\si,\;\st [\phi_y^{\si_1}, a^{\si_2}]\in\oplus_{\si=|\si_1-\si_2|}^{\si_1+\si_2}V_\si.$$ 
			\end{proposition}
			
			It is also important to understand the action of the vierbein form $e$. We define a linear operator:
			\begin{equation}
			\begin{split}
			L:&V_\si\to V_\si,\\
			&a\to \st[e,a],
			\end{split}
			\end{equation}
			which obeys the following properties:
			\begin{proposition}{\rm{\cite[Section 2.3.2]{MazzeoWitten2013}}}
				\label{actionofL}
				(1) $L$ has three eigenspaces $V^-_\si,V^0_\si,V^+_\si$ with dimension $2\si-1,2\si+1,2\si+3$ and eigenvalues $\si+1,1,-\si$. We can write $V_\si=V^-_\si\oplus V^0_\si\oplus V^+_\si.$ 
				
				(2) For 1-form $a\in V_{\si}$, $[\st e,a]=0$ implies $V^0_{\si}$ component of $a$ is zero.
			\end{proposition}
			
			\begin{example}
				\rm{	When $G=SU(2)$ or $SO(3)$, let $\{e_i\}$ be the orthnormal basis of $T^{\star}X$ and $\{\mathfrak{t}_i\}$ be basis of $\gpp$ with acyclic relationship, then there is a unique principle embedding of $\mathfrak{su}(2)\to \mathfrak{g}$, with $\si=1$. When the vierbein form $e=\sum_{i=1}^3e_i\mathfrak{t}_i$, we can explicitly compute that 
					$V^-=span \{e\},\;V^0=span\{\mathfrak{t}_ie_j-\mathfrak{t}_je_i\}_{i\neq j},\;V^+=span\{\mathfrak{t}_1e_1-\mft_2 e_2,\;\mathfrak{t}_1e_1-\mft_3 e_3,\;\mft_ie_j+\mft_je_i\}_{i\neq j}$. This decomposition appears in \cite{leung2018energy}. }
			\end{example}
			
			We write $V^{\circ }:=\oplus_\si V_\si^{\circ}$, where $\circ\in\{+,-,0\}$ and for an integer $k$, we define the following operator $\ML_{k}^{\si}(a):=ka+\st[e,a]$ and obtain
			\begin{corollary}
				$\ML_k^{\si}$ is an isomorphism for $k\neq (\si+1),-1,-\si$. 
				\label{isom}
			\end{corollary}
			\proof Follows directly from Proposition \ref{actionofL}.
			\qed
			
			Let $\om$ be a connection on $X$, for the 3-dimensional differential operator $\st d_{\om}$, which acts from the space $\Omega_X^1(\gpp)$ to itself, has the following properties:
			\begin{proposition}{\rm{\cite[Page 5]{henningson2012boundary}}}$\st d_{\om}:V^-_\si\to V^0_\si,\; V^0_\si\to V^-_\si\oplus V^0_\si\oplus V^+_\si,\;V^+_\si\to V^0_\si\oplus V^+_\si$.
				\label{deaction}
			\end{proposition}
			
			In addition, the vierbein form $e$ will give an identification of $V^0_1$ and $\Omega^0(\tau_1)$. We define the operator
			\begin{equation}
			\begin{split}
			&\Gamma:\Omega^1(\gpp)\to\Omega^0(\gpp),\\
			&\Gamma(a):=\st[a,\st e],
			\end{split}
			\end{equation}
			where $a$ is a 1-form, then we have
			the following identities:
			\begin{proposition}{\rm{\cite[Appendix A]{Mikhaylov17}}}
				For $a\in V_\si=\Omega^1(\tau_\si),b\in \Omega^0(\tau_\si)$, we have:
				
				(1)$\Gamma a=\Gamma (a)^0,\;[e,\Gamma a]=2(a)^0$, $\Gamma[e,b]=2b$, where $(a)^0$ means the projection to the $V^0_\si$ part under the decomposition $V_\si=V^+_\si\oplus V^0_\si\oplus V^0_\si,$
				
				(2)$[e,d^{\st}_\om a]=2(\st d_\om a)^0$, $\Gamma(\st d_\om a)=d^{\st}_\om a$, where $(\st d_\om a)^0$ is the $V^0_\si$ part of $\st d_\om a$.
				
				\label{identity}
			\end{proposition}
			
			With this preparation, we will state two algebraic lemma which will be heavily used in the following parts of the paper:
			\begin{lemma}
				\label{VanishingLemma}
				For $a_1,a_2\in V_\si$, $b\in V_\si^+$, $c_1,c_2\in \Omega^0(\tau_\si)$ and a positive integer $r$, if they satisfy the following algebraic equations
				\begin{equation}
				\begin{split}
				&(\si+1)a_1=\st[e,a_1]-[e,c_1],\\
				&(\si+1)c_1=-\Gamma a_1,\\
				&ra_1+(\si+1)a_2=\st[e,a_2]+[c_2,e]+\st d_\om b,\\
				&rc_1+(\si+1)c_2=-\Gamma a_2+d_\om^\st b,
				\end{split}
				\end{equation}
				then $a_1=0$ and $c_1=0.$
			\end{lemma}
			\begin{proof}
				Consider the $V^+_\si$ part of the first equation, we obtain $(a_1)^+=0$, where $(a_1)^+$ is the $V^+_\si$ part of $a_1$. As $-(\si+1)a_2+\st[e,a_2]\in V^+\oplus V^0$ and by Lemma \ref{deaction}, $\st d_\om b\in V^+\oplus V^0$, consider the $V^-_\si$ part of the third equation, we obtain $(a_1)^-=0$. 
				
				The only situation left is the $V^0$ part. Consider the $V^0$ part of the first two equations, we obtain 
				\begin{equation}
				\si(a_1)^0=-[e,c_1],\;c_1=-\frac{1}{\si+1}\Gamma a_1.
				\label{V0partLemma}
				\end{equation}
				Applying Proposition \ref{identity}, we obtain $\si (a_1)^0=\frac{2}{\si+1}(a_1)^0.$ When $\si\neq 1$, then $(a_1)^0=0$.
				
				If $\si=1$, then the two equations in \eqref{V0partLemma} are equivalent. To be explicit, we obtain 
				$(a_1)^0=-[e,c_1]$ or $c_1=-\frac12 \Gamma a_1.$ We consider the $V^0_\si$ part of the last two equations, we obtain
				\begin{equation}
				\begin{split}
				&r(a_1)^0+(a_2)^0=[c_2,e]+(\st d_\om b)^0\\
				&rc_1+2c_2=-\Gamma a_2+d_\om^\st b.
				\label{V0partlemma2}
				\end{split}
				\end{equation}
				Using $[e,\;]$ acts on the second equation of \eqref{V0partlemma2}, we obtain $$-r(a_1)^0+2(a_2)^0=2[c_2,e]+2(\st d_\om b)^0.$$ Comparing the coefficients with the first equation of \eqref{V0partlemma2}, we obtain $(a_1)^0=0$. 
			\end{proof}
			
			\begin{lemma}
				\label{vanishinglemma2}
				Let $a,\Theta\in V_\si^0$, $\phi,\Xi\in \Omega^0(\tau_\si)$, let $\lam$ be a real number such that $\lam\neq 2$ or $-1$. If they satisfy
				\begin{equation}
				\begin{split}
				(\lam-1)a&=\Theta-[e,\phi],\;\lam \phi=-\Gamma a+\Xi,
				\end{split}
				\end{equation}
				then 
				\begin{equation}
				\begin{split}
				a=\frac{1}{\lam^2-\lam-2}(\lam \Theta-[e,\Xi]),\; \phi=\frac{1}{\lam^2-\lam-2}((\lam-1)\Xi-\Gamma \Theta).
				\end{split}
				\end{equation}
				Specially, if $\Theta=\Xi=0$, then $a=\phi=0$.
			\end{lemma}
			\begin{proof}
				Applying $[e, ]$ to the second equation, by Proposition \ref{identity}, we compute $$\lam[e,\phi]=-[e,\Gamma a]+\Gamma \Xi=-2a+\Gamma \Xi.$$ Combining with the first equation, we obtain
				$a=\frac{1}{\lam^2-\lam-2}(\lam \Theta-[e,\Xi])$. Acting the operator $\Gamma$ to the first equation and combing with the second, we will obtain $ \phi=\frac{1}{\lam^2-\lam-2}((\lam-1)\Xi-\Gamma \Theta)$. The rest follows immediately.
			\end{proof}
			
		\end{subsection}
	\end{section}
	
	\begin{section}{Leading Expansions of the Nahm Pole Solutions}
		In this section, we will determined the coefficients a Nahm pole solution up to several leading terms. For $P$ a principle $G$ bundle over $X\ti\RP$, $A$ a connection over $P$ and $\Phi$ a $\gpp$-valued 1-form. We choose a gauge such that $A$ doesn't have $dy$ component. We write $\Phi=\phi+\phi_y dy$ and recall the follow equation
		\begin{equation*}
		\begin{split}
		&\pa_y A=\st d_A\phi+[\phi_y,\phi],\\
		&\pa_y\phi=d_A\phi_y+\st(F_A-\phi\we\phi),\\
		&\pa_y\phi_y=d_A^{\st}\phi.
		\end{split}
		\end{equation*}
		We denote $A=\om+a,\;\Phi=y^{-1}e+b+\phi_y\;dy$ where $\om$ is a connection independent of the $\RP$ direction, $a,b$ are $1-$forms independent of $y$. Recalling theorem \ref{expansionexists}, we can assume $a,b$ have expansions with leading order $\MO(y(\log y)^{p})$.
		
		By a straight forward computation, the $\MO(y^{-2})$ order coefficients of the flow equations \eqref{flowequation} is $\st e=e\we e$, which is automatically satisfied. The $\MO(y^{-1})$ order coefficients of equations \eqref{flowequation} reduce to 
		$d_{\omega}e=0,\;d_{\omega}^{\st}e=0.$ These equations can be understood as:
		\begin{proposition}{\rm{\cite[Section 4.1]{MazzeoWitten2013}}}
			\label{Levi-Civita}
			Under the pull back induced by the vierbein form $e:TX\rightarrow \gpp$, $\om$ is the Levi-Civita connection of $X$.
		\end{proposition}
		Under the decomposition of Proposition \ref{actionofL}, we can also decompose the dual $\st F_\om$ of the Riemannian curvature
		\begin{equation}
		\st F_\om=(\st F_\om)^-+(\st F_\om)^+,
		\end{equation}
		where the first term is the curvature scale and the second term is the traceless part of the Ricci tensor, where these determine the curvature tenser completely for 3-dimensional manifold. In addition, under the decomposition \eqref{decomposition1form}, $\st F_\om\in V_1,$ where $V_1$ is the component corresponds to the principal embedding.
		\begin{subsection}{Leading Order Expansion}
			Now, we will explicitly study the expansions of the subleading terms. Recall that $A=\om+a,\;\Phi=y^{-1}e+b+\phi_y\;dy$, using \eqref{flowequation}, we obtain the following equations, which is also precisely studied by Taubes in \cite{taubes2018sequences}.
			\begin{equation}
			\begin{split}
			&\partial_y a=\st d_\om b+y^{-1}\st[a,e]+y^{-1}[\phi_y,e]+\st[a,b]+[\phi_y,b],\\
			& \partial_yb=\st F_\om-y^{-1}\st[e,b]+\st d_\om a+d_\om\phi_y+[a,\phi_y]+\st a\wedge a-\st b\wedge b,\\
			&\pa_y\phi_y=d^\st_\om b-y^{-1}\st[a,\st e]-\st[a,\st b].
			\label{se}
			\end{split}
			\end{equation}
			
			\begin{subsubsection}{$\si=1$}
				Recalling \eqref{decomposition1form}, we have the projection map $\MP_{\si}:\Omega^1(\gpp)\to V_{\si}$ and $\MP_{\si}:\Omega^0(\gpp)\to \Omega^0(\tau_\si)$. 
				We define $a^{\si}:=\MP_{\si}a,\;b^{\si}:=\MP_{\si}b,\;\phi_y^\si:=\MP_\si \phi_y$
				By Proposition \ref{Levi-Civita}, $\MP_1\st F_\om=\st F_\om$. Applying $\MP_1$ to both side of \eqref{se}, we obtain:
				\begin{equation}
				\begin{split}
				&\partial_y a^{1}=\st d_\om b^{1}+y^{-1}\st[e,a^{1}]+y^{-1}[\phi_y^1,e]+\MP_1(\st[a,b]+[\phi_y,b]),\\
				& \partial_yb^{1}=\st F_\om-y^{-1}\st[e,b^1]+\st d_\om a^1+d_\om\phi_y^1+\MP_1([a,\phi_y]+\st a\wedge a-\st b\wedge b),\\
				&\pa_y\phi_y^1=d^\st_\om b^1-y^{-1}\st[a^1,\st e]-\MP_1(\st[a,\st b]).
				\label{sigma=1projectionequation}
				\end{split}
				\end{equation}
				
				By Theorem \ref{expansionexists}, we write the expansions of $a^1,b^1,\phi_y^1$ as $$a^1\sim \sum_{p=0}^{r_1}a^1_{1,p}y(\log y)^{p}+\cdots,\;b^1\sim \sum_{p=0}^{r_1}b^1_{1,p}y(\log y)^{p}+\cdots,\;\phi_y^1\sim\sum_{p=0}^{r_1} (\phi_y^1)_{1,p}y(\log y)^p+\cdots.$$
				For simplification, we also denote $a^1_1:=a^1_{1,0}$ and $b^1_1:=b^1_{1,0}.$ Then we obtain:
				
				\begin{proposition}
					\label{Einstein1}
					(a) $a^1_{1,p}=(\phi^1_y)_{1,p}=0$ for any $p$,
					
					(b) For $p\geq 2$, $b^1_{1,p}=0$, 
					
					(c) $b_{1,1}^1$ equals to the $V^+$ part of $\st F_{\om}$. Moreover, there exists $c^+\in V^+$ such that we can express $b^1_1$ in the following way:
					$$b^1_1=c^++\frac12 (\st F_\om)^0+\frac13 (\st F_\om)^-.$$
				\end{proposition}
				\begin{proof}
					(a) Consider the $\MO((\log y)^{r_1})$ order coefficients of the first and third equations of \eqref{sigma=1projectionequation}, we  obtain: 
					\begin{equation}
					\begin{split}
					&a_{1,r_1}^1=\st[e,a_{1,r_1}^1]+[(\phi_y^1)_{1,r_1},e],\\
					&(\phi^1_y)_{1,r_1}=-\st[a^1_{1,r_1},\st e].
					\end{split}
					\end{equation}
					As $[(\phi_y^1)_{1,r_1},e]\in V^0$, by Proposition \ref{actionofL}, $(a_{1,r_1}^1)^+=(a_{1,r_1}^1)^-=0$. Projecting the first equation into $V^0$ part, we obtain $[(\phi_y^1)_{1,r_1},e]=0$. Using Proposition \ref{identity}, we obtain $(\phi_y^1)_{1,r_1}=\Gamma[(\phi_y^1)_{1,r_1},e]=0$. Combing this with $(\phi^1_y)_{1,r_1}=-\st[a^1_{1,r_1},\st e]$, we obtain $(a^1_{1,r_1})^0=0$. By an induction on the integer $r_1$, (1) is proved.
					
					(b) For $r_1\geq 2$, consider the $\MO((\log y)^{r_1})$ and $\MO((\log y)^{r_1})$ coefficients of the second equation of \eqref{sigma=1projectionequation}, we obtain
					\begin{equation}
					\begin{split}
					&b^1_{1,r_1}=-\st[e,b^1_{1,r_1}],\\
					&r_1b^1_{1,r_1}+b^1_{1,r_1-1}=-\st[e,b^1_{1,r_1-1}].
					\end{split}
					\end{equation}
					The first equation implies $b^1_{1,r_1}=0$, where as $b^1_{1,r_1-1}+\st[e,b^1_{1,r_1-1}]\notin V^+$, from the second equation we obtain $b_{1,r_1}^1=0$. 
					
					(c) Consider the $\MO(\log y)$ and $\MO(1)$ coefficient equations for the "$b$" terms, we obtain
					\begin{equation}
					\begin{split}
					&b^{1}_{1,1}+\st[e,b_{1,1}^1]=0,\\
					&b_{1,1}^1+b_1^1+\st[e,b_1^1]=\st F_\om,\\&-a_1^1+\st[e,a_1^1]=0,\;[a_1^1,\st e]=0.
					\label{B11}
					\end{split}
					\end{equation}
					The first equation implies $b_{1,1}^1\in V^+$. From the second equation, we obtain that $b_{1,1}^1$ is determined by the $V^+$ part of $\st F_\om$. However, for any $c^+\in V^+$, $b_1^1+c^+$ also satisfies  $$b_{1,1}^1+b_1^1+c^++\st[e,b_1^1+c^+]=\st F_\om.$$ Note that $b_1^1$ can not be determined by these algebraic equations.
				\end{proof}
				
				Now, suppose the leading orders of $(a,b)$ is $y$ instead of $y(\log y)$, we compute the $\OO(1)$ coefficients of equations $(\ref{se}):b_1^1+\st[e,b_1^1]=\st F_\om$ and obtain the following theorem:
				\begin{theorem}
					\label{Einstein2}
					If the sub-leading term $(a,b)$ of a polyhomegenous solution is $C^1$ to the boundary, then $X$ is an Einstein 3-manifold.
				\end{theorem}
				\begin{proof}
					If $(a,b)$ is $C^1$, we obtain $b_{1,1}^1=0$. By \eqref{B11}, we obtain $b_1^1+\st[e,b_1^1]=\st F_\om$. By Proposition \ref{actionofL}, this implies $F_\om\notin V^+$. In addition, by Proposition \ref{Levi-Civita}, we obtain that $\om$ is the Levi-Civita connection, thus $F_\omega\in V^-\oplus V^0$ which implies $X$ is Einstein.
				\end{proof}
				
				Now, we will determine the next order of the expansions. We have the following descriptions of the "$y^2$" order coefficients:
				\begin{proposition} 
					\label{34prop}
					Under the previous notation, we have :
					
					(a) for $p\geq 0$, $b^1_{2,p}=0$,
					
					(b) for $p\geq 2$, $a^1_{2,p}=(\phi_y^1)_{2,p}=0$,
					
					(c) $(a^1_{2,1})^-=0$, $(a_{2,1}^1)^+=\frac{1}{3}(\st d_\om b^1_{1,1})^+$ and $(a^1_{2,1})^0=\frac{1}{3}(\st d_\om b^1_{1,1})^0$, $(\phi_y^1)_{2,1}=\frac{1}{3}d_\om^\st b^1_{1,1}$,
					
					(d)there exist $c^0\in V^0$, $c^-\in V^-$, such that we can write 
					
					\begin{equation}
					\begin{split}
					&a^1_2=-\frac{1}{9}(\st d_\om b^1_{1,1})^++\frac{1}{3}(\st d_\om b^1_1)^++c^0+c^--\frac 13 (\st d_\om b^1_{1,1})^0+(\st d_\om b^1_1)^0,\\
					&(\phi^1_y)_2=-\frac{1}{2}\Gamma c_0.
					\end{split}
					\end{equation}
				\end{proposition}
				\begin{proof}
					(a) For $r_2\geq 0$, consider the $\MO(y(\log y)^{r_2})$ coefficients of \eqref{se}, we obtain $2b_{2,r_2}^1=-\st[e,b_{2,r_2}^1]$, which implies $b_{2,r_2}^1=0$. By induction, we obtain $b_{2,p}^1=0$ for any $p\geq 0$. 
					
					(b) Now, we will consider the $\phi_y$ and $a$ parts. For $r_2\geq 2$, consider the $\MO(y(\log y)^{r_2})$ and $\MO((y\log y)^{r_2-1})$ coefficients of \eqref{se}, the quadratic terms don't contribute to this order and we obtain \begin{equation}
					\begin{split}
					&2a_{2,r_2}^1=\st[e,a_{2,r_2}^1]-[e,(\phi_y^1)_{2,r_2}],\\
					&2(\phi_y^1)_{2,r}=-\st[a^1_{2,r_2},\st e],\\
					&ra^1_{2,r_2}+2a_{2,r_2-1}^1=\st[e,a^1_{2,r_2-1}]+[(\phi_y^1)_{2,r_2-1},e],\\&r_2(\phi_y^1)_{2,r_2}+2(\phi_y^1)_{2,r_2-1}=-\st[a^1_{2,r_2-1},\st e].
					\end{split}
					\end{equation}
					By Lemma \ref{VanishingLemma}, we obtain $a^1_{2,r_2}=0$. 
					
					For statement (c) and (d), we will first determine the $V^+$ part of the coefficients. 
					
					Consider the $\MO(y\log y)$ coefficients, we obtain
					\begin{equation}
					\begin{split}
					&2a^1_{2,1}=\st d_\om b^1_{1,1}+\st[e,a^1_{2,1}]-
					[e,(\phi^1_y)_{2,1}],\\
					&2(\phi^1_y)_{2,1}=d^{\st}_\om b^1_{1,1}-\Gamma a^1_{2,1}.
					\end{split}
					\label{18}
					\end{equation}
					The $V^+$ projection of the first equation gives $(a^{1}_{2,1})^+=\frac13 (\st d_\om b^1_{1,1})^+$. As $b^1_{1,1}\in V^+$, $d_\om b^1_{1,1}\in V^+\oplus V^0.$ The $V^-$ projection of the second equation gives $(a^{1}_{2,1})^-=0.$  
					
					From the $\MO(y)$ coefficients, we obtain
					\begin{equation}
					\begin{split}
					&2a^1_2+a^1_{2,1}=\st d_\om b^1_1+\st[e,a^1_2]-[e,(\phi^1_y)_2],\\
					&(\phi_y^1)_{2,1}+2(\phi_y^1)_{2}=d_\om^\st b^1_1-\st[a^1_2,\st e].
					\end{split}
					\label{19}
					\end{equation}
					The $V^+$ projection of the first equation gives $a^1_2:=\frac{1}{3}(\st d_\om b^1_1)^+-\frac{1}{9}(\st d_\om b^1_{1,1})^+.$ We cannot determine the $V^-$ part of the $a^1_2$. from these algebraic equations.
					
					Now, consider the $V^0$ projection of \eqref{18} and \eqref{19}, we obtain
					\begin{equation}
					\begin{split}
					&(a_{2,1}^1)^0=(\st d_\om b_{1,1}^1)_0-[e,(\phi_y^1)_{2,1}],\\
					&(a_2^1)^0+(a^1_{2,1})^0=(\st d_\om b^1_1)^0-[e,(\phi_y^1)_2],\\
					&(\phi^1_y)_{2,1}+2(\phi_y^1)_2=d^{\st}_\om b^1_1-\Gamma a^1_2.
					\end{split}
					\end{equation}
					Using $[e,\;]$ acts on the third equation and combing with the first equation, we obtain
					\begin{equation}
					\begin{split}
					&2(a_2^1)^0-(a_{2,1}^1)^0=2(\st d_\om b^1_1)^0-2[e,(\phi_y^1)_2]-(\st d_\om b^1_{1,1})^0,\\
					&2(a_2^1)^0+2(a_{2,1}^1)^0=2(\st d_\om b_1^1)^0-2[e,(\phi_y^1)_2].
					\end{split}
					\end{equation}
					Thus, we obtain $(a_{2,1}^1)^0=\frac{1}{3}(\st d_\om b^1_{1,1})^0$. Using \eqref{18}, we compute
					\begin{equation}
					\begin{split}
					(\phi_y^1)_{2,1}=\frac12 d_\om^\st b^1_{1,1}-\frac 12 \Gamma a^1_{2,1}=\frac{1}{2}d_\om^\st b^1_{1,1}-\frac{1}{6}\Gamma(\st d_\om b^1_{1,1})=\frac 13 d_\om^\st b^1_{1,1}.
					\end{split}
					\end{equation}
					Now, we write $(a^1_2)^0=c^0+(\st d_\om b^1_1)^0-\frac 13 (\st d_\om b^1_{1,1})^0$, then we compute
					\begin{equation}
					\begin{split}
					(\phi_y^1)_2&=-\frac{1}{2}(\phi^1_y)_{2,1}+\frac{1}{2}d^{\st}_\om b^1_1-\frac{1}{2}\Gamma a^1_2\\
					&=-\frac{1}{6}(d^\st_\om b^1_{1,1})^0+\frac12 d_\om^\st b^1_1-\frac{1}{2}\Gamma a^1_2\\
					&=-\frac{1}{2}\Gamma c^0.
					\end{split}
					\end{equation}
				\end{proof}

			\end{subsubsection}
			
			\begin{subsubsection}{$\si>1$}
				Now, we will study the coefficients of the expansions when $\si>1$.
				
				Under the projection $\MP_\si:\Omega^1(\gpp)\to V^\si$ of \eqref{se}, as $\MP_\si\st F_\om=0$ for $\si\neq 1$, we obtain
				\begin{equation}
				\begin{split}
				&\partial_y a^{\si}=\st d_\om b^{\si}+y^{-1}\st[e,a^{\si}]+y^{-1}[\phi_y,e]+\MP_\si(\st[a,b]+[\phi_y,b]),\\
				& \partial_yb^{\si}=-y^{-1}\st[e,b^{\si}]+\st d_\om a^{\si}+d_\om \phi_y+\MP_\si([a,\phi_y]+\st a\wedge a-\st b\wedge b),\\
				&\pa_y \phi_y^\si=d^\st_\om b^{\si}-y^{-1}\st[a^{\si},\st e]-\MP_\si(\st[a,\st b]).
				\label{sigmaprojectionequation}
				\end{split}
				\end{equation}
				
				Let ${\si_1,\cdots,\si_N}$ be the possible indices of \eqref{decomposition1form} with $\si_i>1.$ We assume $a^{\si_i},\;b^{\si_i},\;\phi_y^{\si_i}$ have the expansions with leading terms: 
				\begin{equation}
				\begin{split}
				&a^{\si_i}\sim \sum_{p=1}^{r_i}a_{\lam_i,p}^{\si_i} y^{\lam_i}(\log y)^p+\cdots,\;b^{\si_i}\sim \sum_{p=1}^{r_i}b_{\lam_i,p}^{\si_i} y^{\lam_i}(\log y)^p+\cdots,\\
				&\phi_y^{\si_i}\sim \sum_{p=1}^{r_i}(\phi_y^{\si_i})_{\lam_i,p}y^{\lam_i}(\log y)^p+\cdots.
				\end{split}
				\end{equation} 
				
				Relabel these indices, we can assume $1\leq\lam_1\leq\lam_2\leq \cdots \leq\lam_N$. As we only care about the leading asymptotic behaviors, we can assume $\lam_i\leq \si_i$.
				
				\begin{proposition}
					\label{sib1}
					(a) For $\si>1$, the expansions of $a^{\si}$, $b^{\si}$ and $\phi_y^{\si}$ are
					\begin{equation*}
					a^{\si}\sim y^{\si+1}a^{\si}_{\si+1}+\MO(y^{\si+\frac32}),\;b^{\si}\sim y^{\si}b^{\si}_{\si}+\MO(y^{\si+\frac 12}),\;\phi_y^\si\sim y^{\si+1}(\phi_y^{\si})_{\si+1}+\MO(y^{\si+\frac 32})
					\end{equation*}
					where we write $a^{\si}_{\si+1}:=a^{\si}_{\si+1,0}$, $b^{\si}_{\si+1}:=b^{\si}_{\si+1,0}$ and $(\phi_y^{\si})_{\si+1,0}:=(\phi_y^{\si})_{\si+1}$.
					
					(b) there exist $c^0_{\si}\in V^0_\si,\;c^-_{\si}\in V^-_\si,\;c^+_{\si}\in V^+_\si$, such that we can write the leading coefficients of the expansions of $a^\si,b^\si,\phi_y^\si$ as 
					\begin{equation}
					\begin{split}
					&(\phi_y^\si)_{\si+1}=c^0_{\si_i},\;b^\si_\si=c^+_\si,\\
					&a^\si_{\si+1}=c_\si^-+\frac{1}{2\si+1}(\st d_\om c^+_\si)^++(\st d_\om b^\si_\si)^0+[c^0_\si,e]
					\end{split}
					\end{equation}
				\end{proposition}
				\begin{proof}
					
					We will show $a^{\si_i}_{\lam_i+1,p}=b^{\si_i}_{\lam_i,p}=0$ for any $\lambda_i<\si_i$ and $p\geq 1$. We prove by induction on the index $i$. As when $i=1$, the proof is straight forward and we omit the proof.
					
					Suppose for any $i\leq k-1$, $a^{\si_i}_{\lam_i+1,p}=b^{\si_i}_{\lam_i,p}=(\phi_y^{\si_i})_{\lam_i+1}=0$ for any $\lambda_i<\si_i$ and $p\geq 1$, in other word, $\si_i=\lam_i$. Then for $i=k$, consider the quadratic terms of \eqref{sigmaprojectionequation}, by Cebsch-Gordan coefficients \cite{Kostant59} and the induction assumption, for some positive integer $p$, we have 
					\begin{equation}
					\begin{split}
					&\MP_\si(\st[a,b]+\st[\phi_y,b])\in\MO(y^{\si_k+1}(\log y)^p),\;\MP_\si(\st[a,\st b])\in \MO(y^{\si_k+1}(\log y)^p),\\
					&\MP_\si(\st a\we a-\st b\we b+[a,\phi_y])\in\MO(y^{\si_k+1}(\log y)^p).
					\end{split}
					\end{equation}
					Next, we consider the coefficients of $\MO(y^{\lambda_k-1}(\log y)^p)$ and $\MO(y^{\lam_k}(\log y)^p)$. When $\lam_k\leq \si_k$, the quadratic terms will not influence our following discussions.
					
					Consider the $\MO(y^{\lam_k-1}(\log y)^{r_k})$ and $\MO(y^{\lam_k}(\log y)^{r_k-1})$ order coefficients of the first equation of \eqref{sigmaprojectionequation}, we obtain
					\begin{equation}
					\begin{split}
					\lam_ka^{\si_k}_{\lam_k,r_k}=\st[e,a^{\si_k}_{\lam_k,r_k}]+[(\phi_y^{\si_k})_{\lam_k,r_k},e],\;(\phi_y^{\si_k})_{\lam_k,r_k}=-\Gamma a^{\si_k}_{\lam_k,r_k}.
					\end{split}
					\end{equation} 
					Projecting the first equation into $V^+$ and $V^-$ part, for $\lam_k\leq \si_k$ and any $r_k$, $(a^{\si_k}_{\lam_k,r_k})^{\pm}=0$.
					
					Consider the $V^0$ part of the equation, we obtain
					\begin{equation}
					\begin{split}
					(\lam_k-1)(a^{\si_k}_{\lam_k,r_k})_0=[(\phi_y^{\si_k})_{\lam_k,r_k},e],\;(\phi_y^{\si_k})_{\lam_k,r_k}=-\Gamma a^{\si_k}_{\lam_k,r_k}.
					\end{split}
					\end{equation}
					Applying Lemma \ref{vanishinglemma2}, we obtain $a^{\si_k}_{\lam_k,r_k}=0$.
					
					Consider the $\MO(y^{\lam_k-1}(\log y)^{r_k})$ and $\MO(y^{\lam_k-1}(\log y)^{r_k-1})$ order coefficients of the second equation of  \eqref{sigmaprojectionequation}, for $r_k\geq 1$, we obtain
					\begin{equation}
					\begin{split} \lam_kb^{\si_k}_{\lam_k,r_k}=-\st[e,b^{\si_k}_{\lam_k,r_k}],\;r_kb^{\si_k}_{\lam_k,r_k}+\lam_kb^{\si_k}_{\lam_k,r_{k-1}}=-\st[e,b^{\si_k}_{\lam_k,r_{k-1}}].
					\end{split}
					\end{equation}
					Thus, $\lam_k=\si_k$ and $b^{\si_k}_{\lam_k,r_k}=0$ for $r_k\geq 1$. The $\MO(y^{\si_k})$ order coefficients give $\si_kb^{\si_k}_{\si_k}=-\st[e,b^{\si_k}_{\si_k}]$, thus $b^{\si_k}_{\si_k}\in V^{+}$ and it might be non-zero. 
					
					Now, we can assume $a^{\si_k},\phi_y^{\si_k}$ have the leading expansions 
					\begin{equation}
					\begin{split}
					a^{\si_k}\sim \sum_{p=0}^{r_k}y^{\lam_i+1}(\log y)^p a^{\si_k}_{\lam_i+1,p}+\cdots,\; \phi^{\si_k}_y\sim \sum_{p=0}^{r_k} y^{\si_k+1}(\log y)^p(\phi_y^{\si_k})_{\si_{k}+1,p}  , 
					\end{split}
					\end{equation}
					where for simplification, we also denote the highest order of the $"\log y"$ terms as $r_k$.
					
					For $\lam_k<\si_k$ and $r_k\geq 0$ or $\lam_k=\si_k$ and $r_k>1$, consider the $\MO(y^{\lam_k}(\log y)^{r_k}),\MO(y^{\lam_k}(\log y)^{r_k-1})$ order coefficients of \eqref{sigma=1projectionequation}, we obtain
					\begin{equation}
					\begin{split}
					&(\lam_k+1)a^{\si_k}_{\lam_k+1,r_k}=\st[e,a^{\si_k}_{\lam_k+1,r_k}]+[(\phi_y^{\si_k})_{\lam_k+1,r_k},e],\\
					&(\lam_k+1)\phi^{\si_k}_{\lam_k+1,r_k}=-\st[a^{\si_k}_{\lam_k+1,r_k},\st e],\\
					&r_ka^{\si_k}_{\lam_k+1,r_k}+(\lam_k+1)a^{\si_k}_{\lam_k+1,r_k-1}=\st[e,a^{\si_k}_{\lam_k+1,r_k-1}]+[(\phi_y)^{\si_k}_{\lam_k+1,r_k-1},e],\\
					&r_k(\phi_y^{\si_k})_{\lam_k+1,r_k}+(\lam_k+1)(\phi_y^{\si_k})_{\lam_k+1,r_k-1}=-\st[a^{\si_k}_{\lam_k+1,r_k-1},\st e].
					\end{split}
					\end{equation}
					By Lemma \ref{VanishingLemma}, we obtain $a^{\si_k}_{\lam_k+1,r_k}=(\phi_y^{\si_k})_{\lam_k+1,r_k}=0$.
					
					When $\lam_k=\si_k$ and $r_k=1$, we obtain
					\begin{equation}
					\begin{split}
					&(\si_k+1)a^{\si_k}_{\si_k+1,1}=\st[e,a^{\si_k}_{\si_k+1,1}]+[(\phi_y^{\si_k})_{\lam_k+1,r_k},e],\\
					&(\lam_k+1)\phi^{\si_k}_{\si_k+1,r_k}=-\st[a^{\si_k}_{\si_k+1,r_k},\st e],\\
					&(\si_k+1)a_{\si_k+1}^{\si_k}+a^{\si_k}_{\si_k+1,1}=\st d_\om b^{\si_k}_{\si_k}+\st[e,a^{\si_k}_{\si_k+1}],\\
					&r_k(\phi_y^{\si_k})_{\si_k+1,r_k}+(\lam_k+1)(\phi_y^{\si_k})_{\si_k+1,r_k-1}=-\st[a^{\si_k}_{\si_k+1,r_k-1},\st e].
					\end{split}
					\end{equation}
					Applying Lemma \ref{VanishingLemma}, we obtain $a^{\si_k}_{\si_k+1,r_k}=(\phi_y^{\si_k})_{\lam_k+1,r_k}=0$ for $r_k\geq 1.$ This completes the first half of the proposition. 
					
					For the second half of the proposition, we can assume for $\si>1$ and we have the expansions
					\begin{equation*}
					a^{\si}\sim y^{\si+1}a^{\si}_{\si+1}+\MO(y^{\si+\frac32}),\;b^{\si}\sim y^{\si}b^{\si}_{\si}+\MO(y^{\si+\frac 12}),\;\phi_y^\si\sim y^{\si+1}(\phi_y^{\si})_{\si+1}+\MO(y^{\si+\frac 32}).
					\end{equation*}
					Using \eqref{sigmaprojectionequation}, the leading coefficients satisfies the following equations:
					\begin{equation}
					\begin{split}
					&(\si+1)a^{\si}_{\si+1}=\st d_\om b^\si_\si+\st[e,a^\si_{\si+1}]+[(\phi_y^\si)_{\si+1},e],\\
					&\si b^\si_\si=-\st [e,b^\si_\si],\\
					&(\si+1)(\phi_y^\si)_{\si+1}=d_\om^\st b^\si_\si-\st[a^\si_{\si+1},e].
					\end{split}
					\end{equation}
					The claim follows immediately. 
				\end{proof}

				\begin{remark}
					Even by Proposition \ref{sib1}, the leading terms in the expansions of $a^{\si},b^{\si}$ don't have "$\log y$" terms, the $\log$ terms might still appear in the rest terms of the expansion. By Proposition \ref{ClebschGordan}, the expansions of $a^1,b^1$ will contribute to the expansions of $a^{\si},b^{\si}$ from the quadratic terms and the "$\log y$" terms might come from it.
				\end{remark}
			\end{subsubsection}
		\end{subsection}
		\begin{subsection}{Formal Expressions of Higher Order Terms}
			In this section, we will give formal expressions of Higher order terms. We write the expansions of $a,b,\phi_y$ as $$a\sim \sum a_{k,p}y^k(\log y)^p,\;b\sim\sum b_{k,p}y^k
			(\log y)^p,\;\phi_y\sim\sum b_{k,p}y^k
			(\log y)^p$$ 
			and write $a_{k,p}^{\si}:=\MP_{\si}a_{k,p},\;b_{k,p}^{\si}:=\MP_{\si}b_{k,p}$. For each order of $k$, there will be only finite number of $p$ such that $a^{\si}_{k,p},\;b^{\si}_{k,p}$ and $(\phi_y^\si)_{k,p}$ non-vanishing. We obtain the following proposition:
			\begin{proposition}
				\label{prop36}
				For any integer $\si$, $p\geq 0$ and $k\geq \si+1$, $a^{\si}_{k+1,p},\;b^{\si}_{k,p},\;(\phi_y^\si)_{k,p}$ are determined by $\{a^{\si_i}_{\si_i+1},b^{\si_i}_{\si_i+1},(\phi_y^{\si_i})_{\si_i+1}\}$ for all possible integers $\si_i$ in the decomposition \eqref{decomposition1form}. 
			\end{proposition}
			\proof Using \eqref{sigma=1projectionequation}, \eqref{sigmaprojectionequation}, consider the $\MO(y^k(\log y)^p)$ and $\MO(y^{k-1}(\log y)^p)$ order coefficients, we obtain
			\begin{equation}
			\begin{split}
			&(k+1)a^{\si}_{k+1,p}-\st[e,a^{\si}_{k+1,p}]
			=\st d_\om b^{\si}_{k,p}-(p+1)a^{\si}_{k+1,p+1}+[(\phi_y^\si)_{k+1,p},e]\\
			&+\MP_{\si}\sum_{k_1+k_2=k,\;p_1+p_2=p}(\st[a_{k_1,p_1},b_{k_2,p_2}]+[(\phi_y)_{k_1,p_1},b_{k_2,p_2}]),\\
			&kb^{\si}_{k,p}+\st[e,b^{\si}_{k,p}]=\st d_{\om}a^{\si}_{k-1,p}+d_\om (\phi_y^\si)_{k-1,p}-(p+1)b^{\si}_{k,p+1}\\
			&+\MP_{\si}\sum_{k_1+k_2=k,\;p_1+p_2=p}\st(a_{k_1,p_1}\we a_{k_2,p_2}-b_{k_1,p_1}\we b_{k_2,p_2}+[a_{k_1,p_1},(\phi_y)_{k_2,p_2}]),\\
			&(k+1)(\phi_y^\si)_{k+1,p}=-(p+1)(\phi^\si_y)_{k+1,p+1}+(d_\om^\st b^\si)_{k,p}-\st[a^\si_{k+1,p},\st e]\\
			&-\MP_\si\sum_{k_1+k_2=k,\;p_1+p_2=p}(\st[a_{k_1,p_1},\st b_{k_2,p_2}])
			\label{formalexpression}.
			\end{split}
			\end{equation}
			We can write the left hand side of the first equation as $-\ML^{\si}_{-(k+1)}(a^{\si}_{k+1,p})$ and the left hand side of the second equation as $\ML_{k}^{\si}(b^{\si}_{k,p})$.
			
			When $k\geq \si+1$, by Corollary \ref{isom} and Lemma \ref{vanishinglemma2}, we see $a^{\si}_{k+1,p}$, $b^{\si}_{k,p}$ and $(\phi_y^\si)_{k+1,p}$ are uniquely determined by an algebraic combination of $a_{k+1,p+1}^{\si}$, $(\phi_y^\si)_{k+1,p+1}$ and lower order terms. Inducting on $k$ and $p$, we can complete the proof.  
			\qed
			
			For the $\log y$ terms appear in the expansions, we have the following Proposition:
			\begin{proposition}\label{prop37}
				If $b^1_{1,1}=0$, then the expansions of $a,b$ don't have "$\log y$" terms. To be explicit, for any $p\neq 0$ and any $\si$, we have $a^{\si}_{k,p}=b^{\si}_{k,p}=0$. 
			\end{proposition}
			\proof By Proposition \ref{34prop}, if $b_{1,1}^1=0$, we obtain $a_{2,1}^1=(\phi_y^1)_{2,1}=0$. Combing this with Proposition \ref{sib1}, for any $\si$, $a^{\si},b^{\si}$ will not contain $"\log y"$ terms in their leading orders of the expansions. We will prove the proposition by induction. Suppose for $k$, for any $\sigma$, the expansions for $a^{\si},b^{\si}$ up to order $\MO(y^{k+\frac12})$ don't contains $"\log y"$ terms. Let $r$ by the largest order that for some $\si$, $a^{\si}_{k+1,r}$ or $b^{\si}_{k+1,r}$ non-vanishing. By \eqref{formalexpression}, as by our assumption $a^{\si}_{k+1,r+1}=b^{\si}_{k+1,r+1}=0$, we obtain 
			\begin{equation}
			\begin{split}
			-\ML^{\si}_{-(k+1)}(a^{\si}_{k+1,r})=&\st d_\om b^{\si}_{k,r}+[(\phi_y^\si)_{k+1,r},e]\\
			&+\MP_{\si}\sum_{k_1+k_2=k,\;p_1+p_2=p}(\st[a_{k_1,p_1},b_{k_2,p_2}]+[(\phi_y)_{k_1,p_1},b_{k_2,p_2}]),\\
			\ML_{k+1}^{\si}(b^{\si}_{k+1,r})=&\st d_{\om}a^{\si}_{k,r}\\
			&+\MP_{\si}\sum_{k_1+k_2=k+1,\;p_1+p_2=r}\st(a_{k_1,p_1}\we a_{k_2,p_2}-b_{k_1,p_1}\we b_{k_2,p_2}+[a_{k_1,p_1},(\phi_y)_{k_2,p_2}])\\
			(k+1)(\phi^\si_y)_{k+1,r}&=d^\si_\om b^\si_{k,r}-\st[a^\si_{k,r},\st e]-\MP_\si\sum_{k_1+k_2=k+1,\;p_1+p_2=r}(\st[a_{k_1,p_1},\st b_{k_2,p_2}])
			\end{split}
			\end{equation}
			By the induction assumption, for $r\neq 0$, $b^{\si}_{k,r}=0,\;a^{\si}_{k,r}=0$ and the quadratic term in the previous equations vanish. By Corollary \ref{isom} and Lemma \ref{vanishinglemma2}, for $r\neq 0$, $a^{\si}_{k+1,r}=b^{\si}_{k+1,r}=(\phi_y^\si)_{k+1,r}=0.$
			\qed
			
			Now, we can complete the proof of the first two theorem in Chapter 1:
			
			\rm{Proof of Theorem 1.1}
			\proof The statement (1) follows from Proposition \ref{Einstein1}, \ref{34prop}, \ref{sib1}. The statement (2) follows from Proposition \ref{prop36}.  
			\qed
			
			\rm{Proof of Theorem 1.2}
			\proof The statement follows from Proposition \ref{34prop} and \ref{prop37}.
			\qed

		\end{subsection}
		
		\begin{subsection}{More Restrictions}
			In this subsection, we will provide a global restriction on the coefficients of the expansions. By Theorem 1.1, we can assume $a,b$ have the expansions
			\begin{equation}
			\begin{split}
			a\sim a_{2,1}y^2\log y +a_2y^2+\cdots,\;b\sim b_{1,1}y\log y+b_1 y+\cdots,\;\phi_y\sim (\phi_y)_{2,1}y^2\log y+(\phi_y)_2y^2
			\end{split}
			\end{equation}
			
			The following identity over closed manifold comes from \cite{KapustinWitten2006}. For any 4-manifold M with 3-manifold boundary X, let $P$ be a principle $SU(2)$ bundle and $\gpp$ its adjoint bundle, $A$ a connection on $P$ and $\Phi$ a $\gpp$-valued 1-form, denote $I_+:=(\SA+\SB)^+$ and $I_-:=(\SA+\SB)^-$, where $+(-)$ means the (anti-)self-dual parts of the two form over the 4-manifold $M$. We obtain:
			\begin{proposition}
				\begin{equation}
				\begin{split}
				\int_M \Tr(I_+^2+I_-^2)=\int_M \Tr(F_A\we F_A)+\int_X\Tr(\Phi\we d_A\Phi)
				\end{split}
				\end{equation}
			\end{proposition}
			\begin{proof}
				We compute
				\begin{equation*}
				\begin{split}
				\int_M \Tr(I_+^2+I_-^2)=\int_M\Tr(F_A^2+(\Ph)^4-2F_A\we(\Ph)^2-d_A\Ph\we d_A\Ph).
				\end{split}
				\end{equation*}
				Integrating by parts, we obtain $\int _M\Tr(2F_A\we\Ph^2+d_A\Ph\we d_A\Ph)=\int_X\Tr(\Ph\we d_A\Phi)$. In addition, $\int_M\Tr(\Ph^4)=0$.
			\end{proof}
			
			Denote $M=Y\ti\mathbb{R}^+$, consider $(A,\Ph)$ a Nahm pole solution and convergence $C^{\infty}$ to a flat $G^{\mathbb{C}}$ connection at $y\to\infty$. Let $X_{y}:=X\ti\{y\}\subset X\ti\mathbb{R}^+$. By previous identity, we obtain the following:
			\begin{equation}
			\lim_{y\rightarrow 0}\int_{X_y}\Tr(\Phi\we d_A\Ph)-\int_{X_{\infty}}\Tr(\Phi\we d_A\Ph)+\int_M\Tr(F_A\we F_A)=0.
			\label{Is}
			\end{equation}
			
			In addition, recall the Chern-Simons functional of a connection $A$ over 3-manifold $X$ is $\CS(A):=\int_X \Tr (A\we dA+\frac{2}{3}A\we A\we A)$ and it satisfies $
			\int_M\Tr(F_A\we F_A)=\int_M d \CS(A)=\CS(A|_{X_0})-\CS(A|_{X_{+\infty}}).$
			
			By Proposition \ref{Levi-Civita} and the assumption that $(A,\Phi)$ convergence to $G^{\mathbb{C}}$ flat connection at $y=+\infty$, we obtain that $\int_M\Tr(F_A\we F_A)$ is determined by the Levi-Civita connection and the limit flat connection. In particular, it is bounded and we denote $k:=-\int_M\Tr(F_A\we F_A)+\int_{X_\infty}\Tr(\Phi\we d_A\Ph).$ Combing this with \eqref{Is}, we have the following relationship:
			\begin{equation}
			k=\lim_{y\rightarrow 0}\int_{X_y}\Tr(\Ph\we d_A\Ph),
			\label{Is2}
			\end{equation}
			where $k$ is a finite number.
			
			We obtain the following proposition:
			\begin{proposition}
				For $(A=\om+a,\Phi=y^{-1}e+b)$, a \pol solution with expansions as in \rm{(\ref{poly})}, we have
				\begin{equation}
				\lim_{y\rightarrow 0}\int_{X_y}\Tr(\Phi\we d_A\Ph)=-\lim_{y\rightarrow 0}\int_{X_y}\Tr(e\we \st y^{-1}\partial_y a).
				\label{a1}
				\end{equation}
			\end{proposition}
			\proof As $\Phi$ doesn't have $dy$-component, under the temporal gauge, we compute:
			\begin{equation*}
			\begin{split}
			\int_{X_y}\Tr(\Ph\we d_A\Ph)&=\int_{X_y}\Tr(\Ph\we\sta(F_A-\Ph\we\Ph))\\
			&=-\int_{X_y}\Tr(\Ph\we \st_4(\Ph^2))+\int_{X_y}\Tr(\Ph\we\sta (F_A)).
			\end{split}
			\end{equation*}
			where $\sta$ is the 4-dimensional Hodge star of $M$. 
			
			For the first term, when $y$ small, as $\phi_y\sim \MO(y^2\log y)$, we compute 
			\begin{equation}
			\begin{split}
			\int_{X_y}\Tr(\Ph\we \st_4(\Ph^2))&=\int_{X_y}\Tr(\phi\we \st_4(\Ph^2))+\int_{X_y}\Tr(\phi_ydy\we \st_4 (\Ph^2))\\
			&=\int_{X_y}\Tr(\phi\we \st[\phi,\phi_y])+o(1)\\
			&=\log y\int_{X_y}\Tr(e\we \st[e,(\phi_y)_{2,1}])+\int \Tr(e\we \st[e,(\phi_y)_2])+o(1)
			\end{split}
			\end{equation}
			Using the identity $\st e=e\we e$ over $X$, for any 0-form $c$, we have $$\Tr(e\we\st[e,c])=\Tr(e^3 c-e\we c\we e^2)=0.$$ Thus, $\lim_{y\to 0}\int_{X_y}\Tr(\Ph\we \st_4(\Ph^2))=0.$ 
			
			For the other term, we have
			\begin{equation*}
			\begin{split}
			\int_{X_y}\Tr(\Ph\we\st_4F_A)&=\int_{X_y}\Tr(\Ph\we\sta(d_\om a))\\
			&=-\int_{X_y}\Tr(\Ph\we\st(\py a))\\ 
			&=-\int_{X_y}\Tr(y^{-1}e\we\st(\py a)+b\we\st(\py a))\\
			&=-\int_{X_y}\Tr(y^{-1}e\we\st(\py a))+o(1).
			\end{split}
			\end{equation*}
			The last equality is because $b\sim \MO(y(\log y)^p)$ and $\pa_ya\in \MO((\log y)^p)$ for some $p$.
			\qed
			
			We have the following corollary, which build up the relationship between the expansions of a Nahm pole solution and the instanton number:
			\begin{corollary}
				For the \pol solutions $(A,\Phi)$, we obtain: 
				
				(1) $\int_{X_0}\Tr(e\wedge \st a_{2,1})=0$,
				
				(2) $k=2\int_{X_0}\Tr(e\we \st a_2)$, where $X_0$ is $X\times \{0\}\subset \YI.$
				\label{GR}
			\end{corollary}
			\proof
			By previous computation, the non-vanishing terms of $\Tr (y^{-1} e\we \st \partial_y a)$ will be
			\begin{equation*}
			\Tr (2\log y e\we \st a_{2,1}+e\we\st a_{2,1}+2e\we \st a_2).
			\end{equation*}
			
			By the \pol assumption, we know $a_{2,p}\in\MC^{\infty}(Y)$. Combine this with (\ref{a1}), we know all singular terms should vanish. Thus $\int_{X_0}\Tr(e\wedge \st a_{2,1})=0$ The only remaining term that contributes to the integral is $\int_{X_0}\Tr (e\we\st a_2)$ which verifies the statement (2).
			\qed
			
		\end{subsection}
	\end{section}

	\begin{section}{The Expansions When $G=SO(3)$ or $SU(2)$}
		\begin{subsection}{Formula Expansions}
			In this section, we will determine all the rest terms in the expansion of a Nahm pole solution when $G=SU(2)$ or $SO(3)$. When $G=SU(2)$ or $SO(3)$, in the notation of \eqref{decomposition1form}, we only have $\si=1$ and $\tau_1$ is the only irreducible module. Proposition \ref{Einstein1}, \ref{34prop} still works for this case. We will give a proof Theorem 1.4 by induction.
			
			For $k\geq 1$, suppose $(a,b,\phi_y)$ satisfies \eqref{se} and has the following expansions:
			\begin{equation}
			\begin{split}
			a&\sim \sum_{i=1}^{k}\sum_{p=0}^ia_{2i,p}y^{2i}(\log y)^p+\sum_{p=0}^{r_{2k+1}}a_{2k+1,p}y^{2k+1}(\log y)^p+\cdots,\\
			b&\sim \sum_{i=1}^{k}\sum_{p=0}^i b_{2i-1,p}y^{2i-1}(\log y)^p+\sum_{p=0}^{r_{2k+1}}b_{2k+1,p}y^{2k+1}(\log y)^p+\cdots,\\
			\phi_y&\sim \sum_{i=1}^{k}\sum_{p=0}^i(\phi_y)_{2i,p}y^{2i}(\log y)^p+\sum_{p=0}^{r_{2k+1}}(\phi_y)_{2k+1,p}y^{2k+1}(\log y)^p+\cdots,
			\end{split}
			\label{oddexpan}
			\end{equation}
			where "$\cdots$" means the higher order terms. In addition, we denote 
			\begin{equation}
			\begin{split}
			&\MA_{2k}:=\sum_{i=1}^{k}\sum_{p=0}^ia_{2i,p}y^{2i}(\log y)^p,\;\MB_{2k}:=\sum_{i=1}^{k}\sum_{p=0}^i b_{2i-1,p}y^{2i-1}(\log y)^p,\\&\MC_{2k}:=\sum_{i=1}^{k}\sum_{p=0}^i(\phi_y)_{2i,p}y^{2i}(\log y)^p,
			\end{split}
			\end{equation}
			
			be terms in the expansions of $a$ and $b$ which vanish slower than $\MO(y^{2k+\frac 12})$. 
			
			Recall for any 1-form $\al$, we define $\ML_{2k+1}(\al)=(2k+1)\al+\st[e,\al]$ and obtain:
			
			\begin{proposition}
				\label{inducationeven}
				Assume $a,b$ have the expansions in \eqref{oddexpan}, let $p,s$ be non-negative integers, we have:
				
				(1)For any $p$, $a_{2k+1,p}=(\phi_y)_{2k+1,p}=0$ and for integer $s>k+1$, $b_{2k+1,s}=0$,
				
				(2)For $k+1\geq s\geq 0$, $b_{2k+1,s}=\ML_{2k+1}^{-1}(\Theta^s_{2k+1}-(s+1)b_{2k+1,s+1})$, where $\Theta^{s}_{2k+1}$ depends on $\MA_{2k}$ and $\MB_{2k}$. If $\MA_{2k}$ and $\MB_{2k}$ don't have log terms, for $s\geq 1$, we obtain $b_{2k+1,s}=0$.
			\end{proposition}
			\begin{proof}
				Consider the first equation in \eqref{se}, consider the expansion of order $\MO(y^{2k}(\log y)^{r_{2k+1}})$, the quadratic term $\st[a,b]$ will not contribute as $\MA_{2k},\;\MC_{2k}$ only contains even order terms and $\MB_{2k}$ only contains odd order terms. Thus, we obtain 
				\begin{equation}
				\begin{split}
				&(2k+1)a_{2k+1,r_{2k+1}}=\st[e,a_{2k+1,r_{2k+1}}]+[(\phi_y)_{2k+1,r_{2k+1}},e],\\
				&(2k+1)(\phi_y)_{2k+1,r_{2k+1}}=-\Gamma a_{2k+1,r_{2k+1}}. 
				\end{split}
				\end{equation}
				As $k\geq 1$, by Proposition \ref{actionofL} and Lemma \ref{vanishinglemma2}, we obtain $a_{2k+1,r_{2k+1}}=0$. By induction, we obtain $a_{2k+1,p}=0$ for any $p$. 
				
				For the second equations of \eqref{se}, let $s,s_1,s_2$ be non-negative integers, we define
				\begin{equation}
				\begin{split}
				\Theta^{s}_{2k+1}:&=\st d_\om a_{2k,s}+d_\om (\phi_y)_{2k,s}+\sum_{l=1}^k\sum_{s_1+s_2=s} \st a_{2l,s_1} a_{2k-2l,s_2}\\
				&-\sum_{l=1}^k \sum_{s_1+s_2=s}\st b_{2l-1,s_1}b_{2k-2l-1,s_2}+\sum_{l=1}^k\sum_{s_1+s_2=s} \st [a_{2l,s_1},(\phi_y)_{2k-2l,s_2}],
				\end{split}
				\end{equation}
				where $a_{2k,s},\;b_{2k,s}$ are understood as zero if it don't appears in the coefficients of $\MA_{2k}$ and $\MB_{2k}$.
				
				Consider the coefficients of order $\MO(r^{2k}(\log y)^s)$, we obtain 
				\begin{equation}
				(2k+1)b_{2k+1,s}
				=-\st[e,b_{2k+1,s}]+\Theta^{s}_{2k+1}-(s+1)b_{2k+1,s+1}.
				\end{equation}
				
				Thus, as $k\geq 1$, we obtain $b_{2k+1,s}=\ML_{2k+1}^{-1}(\Theta^s_{2k+1}-(s+1)b_{2k+1,s+1})$. Suppose $s=r_{2k+1}$ and $r_{2k+1}>k+1$, then $\Theta_{2k+1}^{r_{2k+1}}=0$. Thus, $b_{2k+1,r_{2k+1}}=0$ for $r_{2k+1}>k+1$. 
				
				If $\MA_{2k},\MB_{2k}$ don't contain $log$ terms and $s\neq 0$, we obtain $\Theta_{2k+1}^{s}=0$. By induction of $s$, this proves the last claim.
			\end{proof}
			
			There is another type of expansions we need to consider:
			for $k\geq 1$, suppose $(a,b)$ has the following expansions
			\begin{equation}
			\begin{split}
			a&\sim \sum_{i=1}^{k}\sum_{p=0}^ia_{2i,p}y^{2i}(\log y)^p+\sum_{p=0}^{r_{2k+2}}a_{2k+2,p}y^{2k+2}(\log y)^p+\cdots,\\
			b&\sim \sum_{i=1}^{k+1}\sum_{p=0}^i b_{2i-1,p}y^{2i-1}(\log y)^p+\sum_{p=0}^{r_{2k+2}}b_{2k+2,p}y^{2k+2}(\log y)^p+\cdots,\\
			\phi_y&\sim \sum_{i=1}^{k}\sum_{p=0}^i(\phi_y)_{2i,p}y^{2i}(\log y)^p+\sum_{p=0}^{r_{2k+2}}(\phi_y)_{2k+2,p}y^{2k+2}(\log y)^p+\cdots,
			\end{split}
			\label{evenexpan}
			\end{equation}
			where "$\cdots$" means the higher order terms. Similarly, we define 
			\begin{equation}
			\begin{split}
			&\MA_{2k+1}:=\sum_{i=1}^{k}\sum_{p=0}^ia_{2i,p}y^{2i}(\log y)^p,\;\MB_{2k+1}:=\sum_{i=1}^{k+1}\sum_{p=0}^i b_{2i-1,p}y^{2i-1}(\log y)^p,\\
			&\MC\textit{}_{2k+1}:=\sum_{i=1}^{k}\sum_{p=0}^i(\phi_y)_{2i,p}y^{2i}(\log y)^p.
			\end{split}
			\end{equation}

			We obtain the following proposition:
			\begin{proposition}
				\label{Inductionodd}
				Assume $a,b,\phi_y$ have the expansions in \eqref{evenexpan}, let $p,s$ be non-negative integers, we have:
				
				(1)For any $p$, $b_{2k+2,p}=0$ and for integer $s\geq k+2$, $a_{2k+2,s}=(\phi_y)_{2k+2,s}=0$,
				
				(2)For $k+1\geq s\geq 0$, we can write
				\begin{equation}
				\begin{split}
				&a^+_{2k+2,s}=\frac{1}{2k+3}((s+1)a^+_{2k+2,s+1}+(\Theta^s_{2k+2})^+),\\
				&a^-_{2k+2,s}=\frac{1}{2k}((2+1)a^-_{2k+2,s+1}+(\Theta^2_{2k+2})^-),\\
				&a^0_{2k+2,s}=\frac{1}{4k^2+6k}((2k+2)((s+1)a_{2k+2,s+1}+\Theta^s_{2k+2}))-[e,\Xi^s_{2k+2}-(s+1)(\phi_y)_{2k+2,s+1}],\\
				&(\phi_y)_{2k+2,s}=\frac{1}{4k^2+6k}((2k+1)(\Xi^s_{2k+2}-(s+1)(\phi_y)_{2k+2,s+1})-\Gamma((s+1)a_{2k+2,s+1}+\Theta^s_{2k+2})),
				\label{complicate}
				\end{split}
				\end{equation}
				$a_{2k+2,s}=-\ML_{-(2k+2)}^{-1}(\Theta^s_{2k+2}-(s+1)a_{2k+2,s+1})$, where $\Theta^{s}_{2k+2},\Xi_{2k+2}$ depend on $\MA_{2k+1}$, $\MB_{2k+1}$ and $\MC_{2k+1}$. If $\MA_{2k+1}$ and $\MB_{2k+1}$ don't have $\log y$ terms, for $s\geq 1$, we obtain $b_{2k+2,s}=0$.
			\end{proposition}
			\begin{proof}
				Consider the $\MO(y^{2k+1}(\log y)^{r_{2k+2}})$ terms of the second equations of \eqref{se}. As $\MA_{2k+1},\MC_{2k+1}$ only contains even order expansions and $\MB_{2k+1}$ only contains odd order expansions, the quadratic terms doesn't contribute and we obtain $(2k+1)b_{2k+2,r_{2k+2}}+\st [e,b_{2k+2,r_{2k+2}}]$=0, which implies $b_{2k+2,r_{2k+2}}=0$. By induction, we obtain for any $p$, $b_{2k+2,p}=0$.
				
				For a non-negative integer $s$, write
				\begin{equation}
				\begin{split}
				&\Theta_{2k+2}^s:=\sum_{l=1}^k\sum_{s_1+s_2=s}\st[a_{2l,s_1},b_{2k+1-2l,s_2}]+\st[(\phi_y)_{2l,s_1},b_{2k+1-2l,2s}]+\st d_{\om}b_{2k+1}^s\\
				&\Xi_{2k+2}^s:=\sum_{l=1}^k\sum_{s_1+s_2=s}\st[a_{2l,s_1},\st b_{2k+1-2l,s_2}]
				\end{split}
				\end{equation}
				We compute the $\MO(y^{2k+1}(\log y)^{s})$ coefficients and obtain
				\begin{equation}
				\begin{split}
				&(2k+2)a_{2k+2,s}+(s+1) a_{2k+2,s+1}=\Theta^s_{2k+2}+\st[e,a_{2k+2,s}]+[(\phi_y)_{2k+2,s},e],\\
				&(2k+2)(\phi_y)_{2k+2,s}+(s+1)(\phi_y)_{2k+2,s+1}=-\Gamma a_{2k+2,s}+\Xi^s_{2k+2}.
				\end{split}
				\end{equation}
				As $k\geq 1$, applying Lemma \ref{vanishinglemma2}, we obtain \eqref{complicate}.
				
				Suppose $s=r_{2k+1}$ and $r_{2k+1}>k+2$, then $\Theta_{2k+2}^{r_{2k+2}}=\Xi^{r_{2k+2}}_{2k+2}=0$. Also by the assumption of the expansion, we have $a_{2k+2,s+1}=(\phi_y)_{2k+2,s+1}=0$. Thus, $a_{2k+2,r_{2k+2}}=a_{2k+2,r_{2k+2}}=0$ for $r_{2k+2}\geq k+2$. If $\MA_{2k+1},\MB_{2k+1}$ don't contain $log$ terms and $s\neq 0$, we obtain $\Theta_{2k+2}^{s}=0$. By an induction of $s$, this proves the last claim. 
			\end{proof}
			
			Now, we will give a proof for Theorem 1.4:
			
			\rm{Proof of Theorem 1.4}: By Proposition \ref{Einstein1}, \ref{34prop}, \ref{inducationeven}, \ref{Inductionodd}, the result follows immediately.
			\qed
		\end{subsection}

		\begin{subsection}{Examples}
			Now, we will give some examples of the expansions of some Nahm pole solutions to the Kapustin-Witten equations that are related to our theorem:
			\begin{example}{\rm{\cite{he2015rotationally}}}
				\rm{
					Nahm pole solutions on $S^3\times \mathbb{R}^+$.
					\label{NahmpoleS3}
					Equip $S^3$ with the round metric and take $\omega$ be 
					Maurer–Cartan 1-form of $S^3$ and $\om$ satisfies the following relation $d\om=-2\om\we\om$ and $\st\om=\om\we\om$. Denote $y$ the
					coordinate of $\mathbb{R}^+$ and  denote 
					\begin{equation}
					(A,\Phi)=(\frac{6e^{2y}}{e^{4y}+4e^{2y}+1}\omega,\frac{6(e^{2y}+1)e^{2y}}{(e^{4y}+4e^{2y}+1)(e^{2y}-1)}\omega),
					\label{S1}
					\end{equation}
					
					\cite[Theorem 6.2]{he2015rotationally} shows that $(A,\Phi)$ is a Nahm-Pole solution to the Kapustin-Witten equations. In addition, the solutions (\ref{S1}) will converge to the unique flat $\slc$ connection in the cylindrical end of $S^3\times \mathbb{R}^{+}$.}
				
				The expansions of this solution along $y\rightarrow 0$ will be 
				\begin{equation}
				\begin{split}
				A \sim (1-\frac{2}{3}y^2+\frac{2}{9}y^4-\frac{4}{135}y^6+\cdots)\om,\;\Phi\sim (\frac{1}{y}-\frac{1}{3}y-\frac{1}{45}y^3+\frac{58}{945}y^5+\cdots)\om,
				\end{split}
				\end{equation}
				and note that the expansions of $A$ only contain even order terms and the expansions of $\Phi$ only contain odd order terms.
			\end{example} 
			
			\begin{example}{\rm{ \cite{Kronheimer2015Personal}}}
				\rm{Nahm pole solutions on $X\times \mathbb{R}^+$ where $X$ is any hyperbolic three manifold.}
				\label{NahmpoleY3}
				
				Let $X$ be a hyperbolic three manifold equipped with the hyperbolic metric $h$. Consider the associated $PSL(2;\mathbb{C})$ representation of $\pi_1(Y)$, this lifts to $SL(2;\mathbb{C})$ and determines a flat $\slc$ connection $\nabla^{flat}$. Denote by $\nabla^{lc}$ the Levi-Civita connection and by $A^{lc}$ the connection form. Take $i\omega:=\nabla^{flat}-\nabla^{lc}$. Then locally, $\omega=\sum \mathfrak{t_i}e_i^{\star}$ where $\{e_i^{\star}\}$ is an orthogonal basis of $T^{\star}Y$ and $\{\mathfrak{t}_a\}$ are sections of 
				the adjoint bundle $\gpp$ with the 
				relation $[\mathfrak{t}_a,\mathfrak{t
				}_b]=2\epsilon_{abc}\mathfrak{t}_c$. We also have $\star_Y\omega=F_{\nabla^{lc}}$. Therefore, by the Bianchi identity, we obtain $\nabla^{lc}(\star_Y\omega)=0$. Combining $F_{flat}=0$ and the relation $\nabla^{flat}-\nabla^{lc}=i\omega$, we obtain $F_{\nabla^{lc}+i\omega}=0$. Hence $F_{lc}=\omega\wedge\omega$, $\nabla^{lc}\omega=0$.
				
				Take $y$ to be the coordinate of $\mathbb{R}^+$ in $Y^3\times \mathbb{R}^+$, now set $f(y):=\frac{e^{2y}+1}{e^{2y}-1},$ and take $(A,\Phi)=(A^{lc},f(y)\omega)$. We refer \cite[Section 2.3]{He2017} for a record of proof in \cite{Kronheimer2015Personal} that this is a Nahm pole solution to the Kapustin-Witten equations.
				
				As $A^{lc}$ is independent of $y$, the solution has the following expansions:
				\begin{equation}
				\begin{split}
				A\sim A^{lc},\;\Phi\sim (\frac{1}{y}+\frac{1}{3}y-\frac{1}{45}y^3+\frac{2}{945}y^5+\cdots)\om.
				\end{split}
				\end{equation}
			\end{example}
			
			\begin{example}{\rm{\cite{he2017extended,he2018extended}}}\rm{Nahm pole solutions on $\Sigma\times S^1\times \mathbb{R}^+$ where $\Sigma$ is a Riemann surface with genus$>1$}
				In \cite{he2017extended, he2018extended}, it is shown that for $G=SU(n)$, there exists Nahm pole solutions over $\Sigma\times S^1\times \mathbb{R}^+$. Even though it is an existence result and the solution is not explicitly constructed, it still shown in \cite[Section 4.4]{he2017extended} that the solutions contains $\log y$ terms. As $S^1\times \Sigma$ is not an Einstein 3-manifold, it matches our expectations.
			\end{example}
		\end{subsection}
	\end{section}
	\medskip
	
	\bibliographystyle{alpha}
	\bibliography{references}
\end{document}